\documentclass[11pt]{article}

\usepackage{amsmath,amsthm,verbatim,amssymb,amsfonts,amscd, graphicx}
\usepackage{graphics}
\usepackage{tikz-cd}
\theoremstyle{plain}
\newtheorem{theorem}{Theorem}
\newtheorem{corollary}{Corollary}

\newtheorem{remark}{Remark}
\newtheorem{proposition}{Proposition}

\theoremstyle{definition}

\DeclareMathOperator{\ch}{ch}
\DeclareMathOperator{\Hom}{Hom}

\begin{document}

\title{Counting homomorphisms from surface groups to finite groups}
\author{Michael R. Klug}
\maketitle

 \begin{abstract}
We prove a result that relates the number of homomorphisms from the fundamental group of a compact nonorientable surface to a finite group $G$, where conjugacy classes of the boundary components of the surface must map to prescribed conjugacy classes in $G$, to a sum over values of irreducible characters of $G$ weighted by Frobenius-Schur multipliers.  The proof is structured so that the corresponding results for closed and possibly orientable surfaces, as well as some generalizations, are derived using the same methods.  We then apply these results to the specific case of the symmetric group.  
 \end{abstract}

\section{Introduction}

 Given a closed nonorientable surface, a result of Frobenius and Schur \cite{frobenius_schur} counts the number of homomorphisms from the fundamental group of the surface to a finite group $G$ in terms of dimensions of the irreducible complex representations of $G$, the Frobenius-Schur indicators of these representations, the order of $G$, and the nonorientable genus of the surface.  Given a closed orientable surface, a result of Mednykh \cite{mednykh} counts the number of homomorphisms from the fundamental group of the surface to a finite group $G$ in terms of the dimensions of the irreducible complex representations of $G$, the order of $G$, and the genus of the surface.  Given a compact oriented surface with labeled boundary components, a finite group $G$, and a choice of conjugacy class in $G$ for each boundary component of the surface, Dijkgraaf and Witten \cite{dijkgraaf_witten} as well as Freed and Quinn \cite{freed_quinn} gave formulas (in the context of mathematical physics) of a formula for the number of homomorphisms from the fundamental group of the surface to $G$, sending the conjugacy classes of the boundaries to the chosen conjugacy classes in terms of the dimensions of the irreducible complex representations of $G$, the values of the characters of these representations on the chosen conjugacy classes, the order of $G$, the sizes of the conjugacy classes, and the genus of the surface.  An elementary character theoretic proof of this was supplied by Zagier \cite{zagier_appendix}. In the case where the genus of the surface is zero, the result was also known to Frobenius \cite{frobenius}.  

 Theorem \ref{thm:main} is the missing case where the surface is compact and nonorientable with labeled boundary components.    In this paper, we supply simple character-theoretic proofs of all of these results following the treatments given by Zagier \cite{zagier_appendix} and Mulase and Yu (see section 4 of \cite{mulase_yu}).  We then draw some topological and combinatorial conclusions from these results.  For a more extensive overview of the history of these equations in the case of closed surfaces, see \cite{snyder_mednykh}.

 In section \ref{sec:algebra}, we discuss a framework for understanding all of these equations as consequence of character-theoretic ideas.  In section \ref{sec:topology}, we make the connection to 2-dimensional topology and in section \ref{sec:combinatorics}, we show how these equations can be used to deduce some symmetric function identities, following Stanley \cite{stanley}.

 \section{Counting solutions to equations in groups} \label{sec:algebra}

All of our representations are over $\mathbb{C}$.  Let $r$ be positive integer and $\gamma \in F_r$.  Then for any finite group $G$ we have the class function $f_{\gamma} = f_{\gamma, G}$ with 
$$
f_\gamma(w) = |\{(g_1,...,g_r) \in G^r : \gamma(g_1,...,g_r) = 1  \}|
$$
where $\gamma(g_1,...,g_r)$ is the image in $G$ of the homomorphism from $F_r \to G$, given by sending the $i$th generator of $F_r$ to $g_i$.  Since $f_\gamma$ is a class function, we have 
$$
f_\gamma = \sum_\chi a_\chi^\gamma \chi
$$
for some coefficients $a_\chi^\gamma \in \mathbb{C}$.  Two basic questions are: When is $f_\gamma$ a character for all $G$ (i.e., all of the $a_\chi^\gamma$ are nonnegative integers) and when is $f_\gamma$ a virtual character for all $G$ (i.e., all of the $a_\chi^\gamma$ are integers)?    

In general, we have
$$
a_\chi = \frac{1}{|G|}\sum_{(g_1,...,g_r) \in G^r} \overline{\chi}(\gamma(g_1,...,g_r)
$$
since by taking the inner product on class functions given by
$$
 \left<f_1, f_2\right> = \frac{1}{|G|} \sum_{g \in G} f_1(g) \overline{f_2(g)}
$$
for which the characters $\chi$ form an orthonormal basis, and thus
 \begin{align*}
	 a_\chi &= \frac{1}{|G|}\sum_{g \in G} f_\gamma(g) \overline{\chi}(g) \\
	        &= \frac{1}{|G|} \sum{g \in G} \sum_{(g_1,...,g_r \in G^r \text{ with } \gamma(g_1,...,g_r) = 1} \overline{\chi}(g) \\
		&= \frac{1}{|G|} \sum_{(g_1,...,g_r) \in G^r} \overline{\chi}(\gamma(g_1,...,g_r)
 \end{align*}
 This has been observed before, for example, in Proposition 3.1 in \cite{parzanchevski_schul}.  However, in special cases, such as $x_1^2x_2^2 \cdots x_k^2 \in F_k$ and $[x_1,y_1]\cdots[x_g,y_g] \in F_{2g}$, the coefficients $a_\chi^\gamma$ have a much more explicit form that we now discuss (see \cite{parzanchevski_schul} for additional results on computing $a^\gamma_\chi$ for specific types of elements $\gamma$).  As a word of warning, note that $f_\gamma$ depends on ``where $\gamma$ is'', for example, $f_{x_1}$ depends on if $x_1 \in F_1$ or in $F_2$. 

Given a irreducible character $\chi$ of a finite group $G$, the Frobenius-Schur indicator of $\chi$ is 
$$
 \nu(\chi) = \frac{1}{|G|}\sum_{g \in G} \chi(g^2)
$$
 For all $\chi$, $\nu(\chi) \in \{ -1, 0, 1\}$ with $\nu(\chi) \neq 0$ if and only if there is a nonzero $G$-invariant bilinear form on the representation with character $\chi$, $\nu(\chi) = 1$ if and only if there exists a symmetric such form, and $\nu(\chi) = -1$ if and only if there exists a skew-symmetric such form \cite{frobenius_schur}.   At the root of our discussion is the following theorem of Frobenius \cite{frobenius} (for the first equation) and Frobenius and Schur \cite{frobenius_schur} (for the second equation):

 \begin{theorem}\label{fs}(Frobenius, Schur) 
	Let $G$ be a finite group and $w \in G$ an element.  Then
$$
	 |\{(x,y) \in G \times G : xyx^{-1}y^{-1} = w \}| = \sum_{\chi} \left(\frac{|G|}{\chi(1)} \right) \chi(w)
$$
and
$$
	 |x \in G : x^2	 = w \}| = \sum_{\chi} \nu(\chi) \chi(w)
$$
where the sums are over the irreducible characters of $G$.  
\end{theorem}

Equivalently, Theorem \ref{fs} says that for every irreducible character $\chi$, we have
	 $$
	a^{xyx^{-1}y^{-1}}_\chi = \frac{|G|}{\chi(1)}
	 $$
for $xyx^{-1}y^{-1}$ in a free group generated by $x$ and $y$ and 
	 $$
	a^{x^2}_\chi = \nu(\chi)
	 $$
for $x^2$ in the free group generated by $x$.  

The following result (see \cite{stanley} Exercise 7.69 (d)) shows how to obtain an expression for $a_\chi^{\gamma_1\cdots \gamma_m}$ in terms of $a_\chi^{\gamma_1},..., a_\chi^{\gamma_m}$ when the words $\gamma_1, ..., \gamma_m$ contain disjoint letters.

 \begin{proposition} \label{prod}
	 Let $G$ be a finite group and let $f_1,...,f_m$ be class functions on $G$.  Define the class function $F = F_{f_1,...,f_m}$ by
	 $$
	 F(w) = \sum_{u_1\cdots c_m = w} f_1(u_1) \cdots f_m(u_m)
	 $$
	 Let $\chi$ be an irreducible character of $G$.  Then
	 $$
	 \left< F, \chi \right>  = \left( \frac{|G|}{\chi(1)} \right)^{m-1} \left< f_1, \chi \right> \cdots \left< f_m, \chi \right>
	 $$
 \end{proposition}

Using this, we note that 
$$
f_{[x_1,y_1]\cdots[x_g,y_g]} = F_{f_{[x_1,y_1]},...,f_{[x_g,y_g]}}
$$
and 
$$
f_{x_1^2x_2^2 \cdots x_k^2} = F_{f_{x_1^2},...,f_{x_k^2}}
$$
Thus, from Theorem \ref{fs} together with Proposition \ref{prod}, we obtain the following formula of Mednykh \cite{mednykh} for the first equation and Frobenius and Schur \cite{frobenius_schur} for the second equation - see also \cite{mulase_yu} where this phrasing is used:

\begin{corollary} \label{MulaseYu} (Mednykh, Frobenius,Schur) 
Let $G$ be a finite group and $w \in G$ an element.  For all integers $g, k \geq 0$, we have 
$$
	f_{[x_1,y_1]\cdots[x_g,y_g]}(w) = \sum_{\chi}\left( \frac{|G|}{\chi(1)} \right)^{2g-1} \chi(w)
$$
for $[x_1,y_1]\cdots[x_g,y_g]$ in the free group of rank $2g$ with generators $x_1,y_1,...,x_g,y_g$ and 
$$	
	f_{x_1^2x_2^2 \cdots x_k^2}(w) = \sum_{\chi} \nu(\chi)^k \left( \frac{|G|}{\chi(1)} \right)^{k-1} \chi(w)
$$
for $x_1^2x_2^2 \cdots x_k^2$ in the free group of rank $k$ generated by $x_1,...,x_k$.  
\end{corollary}

\begin{remark}
	Note that setting $w = 1$ in Corollary \ref{MulaseYu}, we see that $|G|$ divides 
$$
	 |\{(x,y) \in G \times G : xyx^{-1}y^{-1} = 1 \}|
$$
and 
$$
	 |\{(x,y) \in G \times G : xyx^{-1}y^{-1} = 1 \}|
$$
	for $g \geq 1$ and $k > 1$.  This is in fact a special case of a theorem of Solomon \cite{solomon} that implies that for any $\gamma \in F_r$ with $r > 1$, $|G|$ divides $f_{\gamma,G}(1)$.  
\end{remark}

We now mention a few more character-theoretic results that we use in the sequel (for a proof see, for example, chapter 3 of \cite{isaacs}). 

\begin{proposition} \label{pre_magic}
Let $\chi$ be an irreducible character of a finite group $G$.  Then the map
	\begin{align*}
		\omega_\chi : Z(\mathbb{C}G) &\to \mathbb{C} \\
				X = \sum_{g \in G }X_g \cdot g &\mapsto \frac{\sum_{g \in G}X_g\chi(g)}{\chi(1)}
	\end{align*}
is an algebra homomorphism where $Z(\mathbb{C}G)$ is the center of the group algebra.  
\end{proposition}

\begin{proposition} \label{magic}
Let $C_1,...,C_n$ be not necessarily distinct conjugacy classes in a finite group $G$ and let $\chi$ be an irreducible character of $G$.  Then
$$
	\chi(1)^{n-1} \sum_{(c_1,...,c_n) \in C_1 \times \cdots \times C_n } \chi(c_1 \cdots c_n) = |C_1| \cdots |C_n| \chi(C_1) \cdots \chi(C_n)
$$
	where $\chi(C_i)$ denotes the value of $\chi$ on an element of $C_i$.  
\end{proposition}

\begin{proof}
Let $C_i^+$ be
$$
	C_i^+ = \sum_{g \in C_i} g
$$
	and note that $C_i^+$ is in the center $Z(\mathbb{C}G)$.  Then, by applying the algebra homomorphism $\omega_\chi$ from Proposition \ref{pre_magic} to the product $C_1^+ \cdots C_n^+$, we obtain
\begin{align*}
	\chi(1)^{-1} \sum_{(c_1,...,c_n) \in C_1 \times \cdots \times C_n } \chi(c_1 \cdots c_n) &= \omega_\chi(C_1^+ \cdots C_n^+)     \\
												 &= \omega_\chi(C_1^+) \cdots \omega_\chi(c_n^+) \\ 
												 &=  \chi(1)^{-n} |C_1| \cdots |C_n| \chi(C_1) \cdots \chi(C_n)
\end{align*}
as desired.  
\end{proof}

We are now prepared to prove the following result that we need in the next section.  The statement of the result is much more natural from a topological perspective, which we discuss in the next section.

\begin{theorem} \label{thm:general}
Let $\gamma \in F_r$, $G$ be a finite group, and $C_1,...,C_n$ be a not necessarily distinct conjugacy classes of $G$.  Then
	\begin{align*}
		|\{ (g_1,...,g_r,c_1,...,c_n) \in G^r \times C_1 \times &\cdots \times C_n : \gamma(g_1,...,g_n)c_1 \cdots c_n = 1 \}|  \\ &= |C_1|\cdots |C_n| \sum_\chi a_{\overline{\chi}}^\gamma \cdot \chi(1)^{1-n}\chi(C_1) \cdots \chi(C_n)
	\end{align*}

\end{theorem}

\begin{proof}
	\begin{align*}
		|\{ (g_1,...,g_r,c_1,...,c_n) \in G^r \times C_1 \times &\cdots \times C_n : \gamma(g_1,...,g_n)c_1 \cdots c_n = 1 \}|  \\ &= \sum_{(c_1,...,c_n) \in C_1 \times \cdots \times C_n } f_\gamma( (c_1 \cdots c_n)^{-1}) \\
					      &=  \sum_{(c_1,...,c_n) \in C_1 \times \cdots \times C_n } a_\chi^\gamma  \cdot \chi((c_1 \cdots c_n)^{-1})) \\
					      &= \sum_\chi a_{\overline{\chi}}^\gamma \sum_{(c_1,...,c_n) \in C_1 \times \cdots \times C_n } \chi(c_1 \cdots c_n) \\
					      &= |C_1| \cdots |C_n| \sum_\chi a_{\overline{\chi}}^\gamma \cdot \chi(1)^{1-n} \chi(C_1) \cdots \chi(C_n)
	\end{align*}
	where the removal of the inverse in the third equation is justified by resumming over the complex conjugates of the characters and recalling that 
$$
\overline{\chi}(g^{-1}) = \chi(g)
$$
\end{proof}

Thus, if we have a more explicit formula for $a_\chi^\gamma$, we obtain a more explicit formula for the expression in the left hand side of Theorem \ref{thm:general}.  We now do exactly that.

\section{Relationship with 2-dimensional topology} \label{sec:topology}

In this section, we demonstrate how the following result is proven and how it relates to topology:
 
 \begin{theorem} \label{thm:main}
Let $C_1, ..., C_n$ be a collection of not necessarily distinct conjugacy classes in a finite group $G$.  Then
	 \begin{align*}
		 |\{(g_1,...,g_k,c_1,...,c_n) \in G^k \times C_1 \times &\cdots \times C_n : a_1^2 \cdots a_k^2 c_1 \cdots c_n = 1  \}| \\  &= |G|^{k-1}|C_1| \cdots |C_n| \sum_{\chi}  \nu(\chi)^k \frac{\chi(C_1) \cdots \chi(C_n)}{\chi(1)^{n + k - 2}}
	 \end{align*}
	 where the sum is over the characters of the irreducible complex representations of $G$ and where $\nu(\chi)$ denotes the Frobenius-Schur indicator of $\chi$. 
\end{theorem}

Let $S_g$ be a closed orientable surface of genus $g$ and let $N_k$ be a closed nonorientable surface of nonorientable genus $k$.  Let $G$ be a finite group.  Noting that 
$$
 \pi_1(S_g) = \left< a_1,b_1,...,a_g,b_g | [a_1,b_1] \cdots [a_g,b_g] \right> 
$$
and 
$$
 \pi_1(N_k) = \left< a_1,...,a_k | a_1^2 \cdots a_k^2 \right> 
$$
we have 
$$
 |\Hom(\pi_1(S_g), G)| = |\{(g_1,...,g_{2g}) \in G^{2g} : [g_1,g_2] \cdots [g_{2g-1}, g_{2g}] = 1  \}|
$$
and
$$
 |\Hom(\pi_1(N_k), G)| = |\{(g_1,...,g_k) \in G^k : g_1^2 \cdots g_k^2 = 1  \}|
$$
Thus, we can rewrite the formula of Mednykh in Theorem \ref{MulaseYu} with $w=1$ as 
$$
|G|^{1-2g} |\Hom(\pi_1(S_g), G)| = \sum_\chi \chi(1)^{2-2g}
$$
and similarly, we can rewrite the formula of Frobenius and Schur in Theorem \ref{MulaseYu} with $w=1$ as 
$$
 |G|^{1-k} |\Hom(\pi_1(N_k), G)| = \sum_\chi (\nu(\chi) \chi(1))^{2-k}
$$

 Let $S_{g,n}$ denote the compact surface of genus $g$ with $n$ boundary components labeled from 1 to $n$.  Further, fix an orientation on $S_{g,n}$ which thus induces an orientation on all of the boundary components of $S_{g,n}$.  Let $C_1,...,C_n$ be a choice of $n$ not necessarily distinct conjugacy classes in $G$.  Let 
 $$
 \Hom^{(C_1,...,C_n)}( \pi_1(S_{g,n}), G)
 $$
 denote the set of homomorphisms from $\pi_1(S_{g,n})$ to $G$ such that the conjugacy class given by the $i$th boundary component using the given orientation is sent to the conjugacy class $C_i$ for $1 \leq i \leq n$. Noting that 
 $$
 \pi_1(S_{g,n}) = \left< a_1,b_1,...,a_g,b_g, c_1,...,c_n | [a_1,b_1] \cdots [a_g,b_g]c_1\cdots c_n \right> 
 $$
we have that 
$$
 |\Hom^{(C_1,...,C_n)}(\pi_1(S_{g,n}), G)| = |\{(g_1,...,g_{2g},c_1,...,c_n) \in G^{2g} \times C_1 \times \cdots \times C_n : [g_1,g_2] \cdots [g_{2g-1}, g_{2g}]c_1 \cdots c_n = 1  \}|
$$
Frobenius \cite{frobenius} proved for $g = 0$, and Dijkgraaf and Witten \cite{dijkgraaf_witten} as well as Freed and Quinn \cite{freed_quinn} proved in general that for $n \geq 1$ that
 \begin{equation} \label{eq:closed_bound}
	 |G|^{1-2g} |\Hom^{(C_1,...,C_n)}(\pi_1(S_{g,n}), G)| = |C_1| \cdots |C_n| \sum_\chi \frac{\chi(C_1)\cdots \chi(C_n)}{\chi(1)^{n + 2g -2}}
 \end{equation}
 where here $\chi(C_i)$ denotes the value of $\chi$ on any element in $C_i$.  Note that by setting $n = 0$, this gives exactly Mednykh's formula in Corollary \ref{MulaseYu}. 

We now demonstrate the use of Theorem \ref{thm:general} by giving a short proof of this formula.  For another elementary proof, see \cite{zagier_appendix}.  

\begin{proof} (of Equation \ref{eq:closed_bound})
 Let $\gamma = [x_1,y_1] \cdots [x_g,y_g]$ in the free group $F_{2g}$ generated by $x_1,y_1,...,x_g,y_g$.  By Corollary \ref{MulaseYu}, we have 
 $$
 a_\chi^\gamma = \left( \frac{|G|}{\chi(1)} \right)^{2g-1}
 $$
 for all irreducible representations $\chi$ of $G$.  Therefore by Theorem \ref{thm:general} together with the observation that
 $$
	a_{\overline{\chi}}^\gamma = a_\chi^\gamma
 $$
 for all $\chi$, the result follows. 
\end{proof}

 This formula shows that the ordering of the conjugacy classes as $(C_1,...,C_n)$ does not affect $|\Hom^{(C_1,...,C_n)}(\pi_1(S_{g,n}), G)|$. This also follows directly without the use of the formula, as noted in \cite{zagier_appendix}, since by using the identity
 $$
c_i c_{i+1} = (c_{i+1})(c_{i+1}^{-1}c_ic_{i+1}) 
 $$
we have a bijection between the set 
$$
\{(g_1,...,g_{2g},c_1,...,c_n) \in G^{2g} \times C_1 \times \cdots \times C_n : [g_1,g_2] \cdots [g_{2g-1}, g_{2g}]c_1 \cdots c_n = 1  \}
$$
and the respective set 
$$
\{(g_1,...,g_{2g},c_1,...,c_n) \in G^{2g} \times C_1 \times \cdots C_{i+1} \times C_i \times \cdots \times C_n : [g_1,g_2] \cdots [g_{2g-1}, g_{2g}]c_1 \cdots c_{i+1} c_i \cdots c_n = 1  \}
$$
given by interchanging the order of the conjugacy classes $C_{i}$ and $C_{i+1}$.  Thus, as far as the number $|\Hom^{(C_1,...,C_n)}(\pi_1(S_{g,n}), G)|$ is concerned, we only need to know $(C_1, ..., C_n)$ as a multiset.


\begin{figure}  	
	\centering
	\includegraphics[width=6cm]{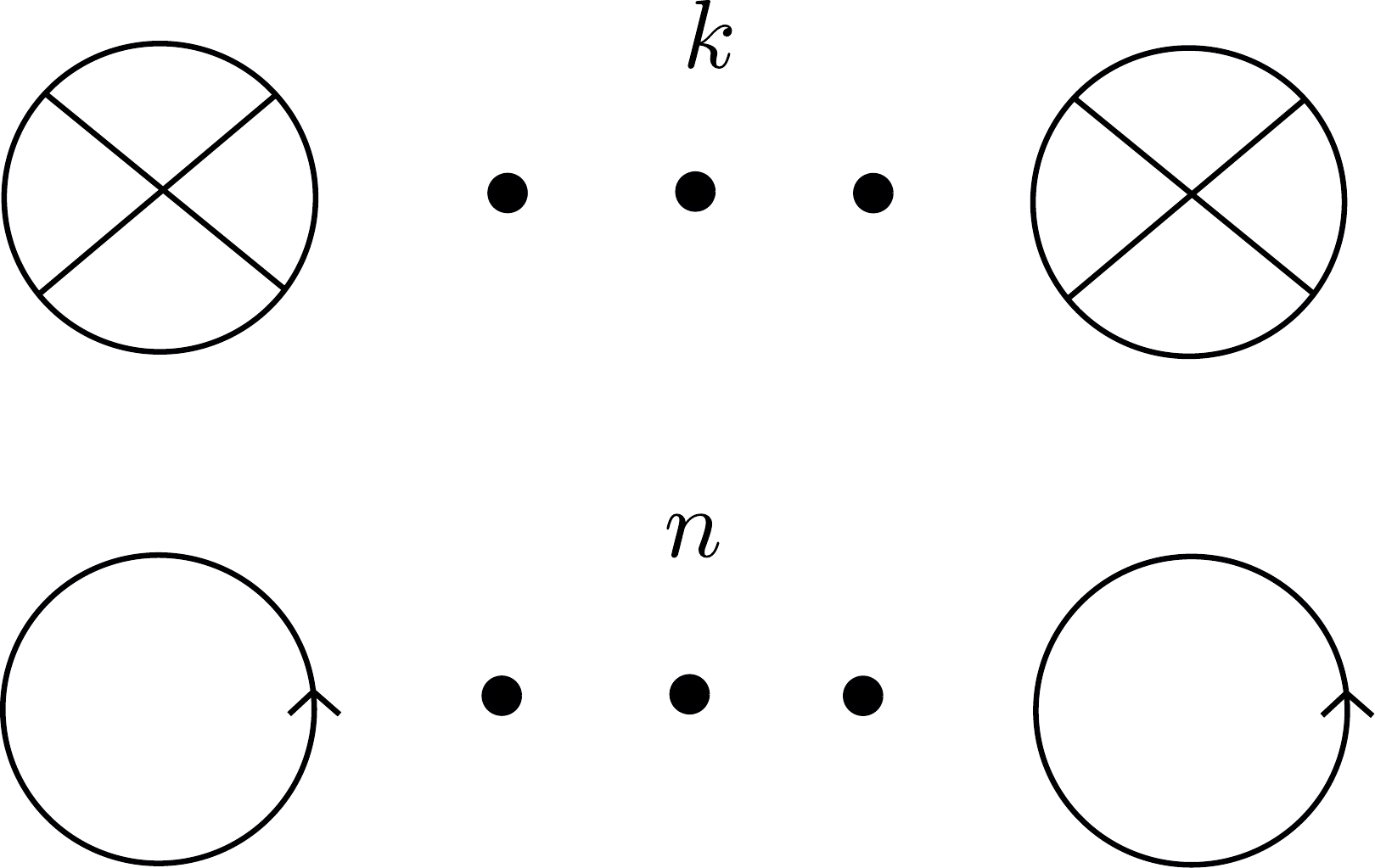}
	\caption{This figure shows the nonorientable compact surface $N_{n,k}$ of nonorientable genus $k$ with $n$ boundary components.  The picture is of a sphere (drawn as just the plane) with $k$ circles where the disks that these circles bound are removed (the lower circles) and $n$ circles (those marked with an X) where the disks that these circles bound have been removed and replaced with M\"obius bands.  The boundary components are the lower circles and they have been given a specific orientation.}
	\label{fig:nonorientable}
\end{figure}

Let $N_{k,n}$ be the compact nonorientable surface of nonorientable genus $k$ with $n$ boundary components labeled with the numbers 1 through $n$  and oriented as in Figure \ref{fig:nonorientable}.  Let 
 $$
 \Hom^{(C_1,...,C_n)}( \pi_1(N_{k,n}), G)
 $$
 denote the set of homomorphisms from $\pi_1(N_{n,k})$ to $G$ such that the conjugacy class given by the $i$th boundary component using the given orientation is sent to the conjugacy class $C_i$ for $1 \leq i \leq n$. Noting that 
 $$
 \pi_1(N_{n,k}) = \left< a_1,...,a_k, c_1,...,c_n | a_1^2 \cdots a_k^2c_1\cdots c_n \right> 
 $$
we have that 
$$
 |\Hom^{(C_1,...,C_n)}(\pi_1(S_{g,n}), G)| = |\{(g_1,...,g_k,c_1,...,c_n) \in G^k \times C_1 \times \cdots \times C_n : a_1^2 \cdots a_k^2 c_1 \cdots c_n = 1  \}|
$$
Just as above in the orientable case, we see (with or without the help of Theorem \ref{thm:main}) that this does not depend on the number of these homomorphisms and does not depend on the given ordering of the conjugacy classes $C_1,...,C_n$.  

Thus, Theorem \ref{thm:main} can be restated as 
$$
	|G|^{1-k} |\Hom^{(C_1,...,C_n)}(\pi_1(N_k), G)| = |C_1| \cdots |C_n| \sum_{\chi}  \nu(\chi)^k \frac{\chi(C_1) \cdots \chi(C_n)}{\chi(1)^{n + k - 2}}
$$
The proof of Theorem \ref{thm:main} is analogous the above proof in the orientable case.

\begin{proof}(of Theorem \ref{thm:main})
 Let $\gamma = x_1^2 x_2^2 \cdots x_k^2$ in the free group $F_{2g}$ generated by $x_1,x_2,...,x_k$.  By Corollary \ref{MulaseYu}, we have 
$$
 a_\chi^\gamma = \nu(\chi)^k \left( \frac{|G|}{\chi(1)} \right)^{k-1} 
$$
 for all irreduicble representations $\chi$ of $G$.  Therefore, by Theorem \ref{thm:general} together with the observation that $\nu(\overline{\chi}) = \nu(\chi)$ and thus
 $$
	a_{\overline{\chi}}^\gamma = a_\chi^\gamma
 $$
 for all $\chi$, the result follows. 
\end{proof}

Note that, by setting $n = 0$, this again gives Frobenius and Schur's result in Corollary \ref{MulaseYu} (though this is not another proof of that result, just an observation).

Note that $\pi_1(S_{g,n})$ and $\pi_1(N_{k,n})$ are both free groups with 
$$
\pi_1(S_{g,n}) = F_{2g + n - 1}
$$
and 
$$
\pi_1(N_{k,n}) = F_{k + n - 1}
$$
Note that for any finite group $G$ we have 
$$
|\Hom(F_m, G)| = |G|^m
$$
Using this fact and summing over the possible tuples of conjugacy classes in equation (\ref{eq:closed_bound}) and Theorem \ref{thm:main}, we obtain:

\begin{corollary}
Let $G$ be a finite group and let $g$ and $n$ be positive integers. Then
$$
	|G|^{n} = \sum_{(C_1,...,C_n)} |C_1| \cdots |C_n| \sum_{\chi}  \frac{\chi(C_1) \cdots \chi(C_n)}{\chi(1)^{n + 2g - 2}}
$$
and
$$
	|G|^{n} = \sum_{(C_1,...,C_n)} |C_1| \cdots |C_n| \sum_{\chi}  \nu(\chi)^k \frac{\chi(C_1) \cdots \chi(C_n)}{\chi(1)^{n + k - 2}}
$$
where the first sums are over all ordered $n$-tuples of conjugacy classes in $G$ and the second sums are over all irreducible complex characters of $G$.  
\end{corollary}

Then, taking the limit as $g \to \infty$ and $k \to \infty$, and noting that any bilinear form on a 1-dimensional space is automotically symmetric and therefore 1-dimensional representations never have Frobenius-Schur indicator equal to -1, we have the following:

\begin{corollary}
Let $G$ be a finite group and let $n$ be a positive integer.  
	\begin{align*}
		|G|^n &= \sum_{(C_1,...,C_n)} |C_1| \cdots |C_n| \sum_{\chi \text{ with } \chi(1) = 1}  \chi(C_1) \cdots \chi(C_n) \\
		      &= \sum_{(C_1,...,C_n)} |C_1| \cdots |C_n| \sum_{\chi \text{ with } \chi(1) = 1 \text{ and } \nu(\chi) =1}  \chi(C_1) \cdots \chi(C_n)
	\end{align*}
where the first sums are over all ordered $n$-tuples of conjugacy classes $(C_1,...,C_n)$ in $G$ and the second sums are over all irreducible complex characters of $G$ that satisfy the specified conditions.
\end{corollary}

As remarked in \cite{zagier_appendix} in the orientable case, Equation (\ref{eq:closed_bound}) and Theorem \ref{thm:main} have a topological interpretation.  In the orientable case, fill in each of the $n$ boundary components of $S_{g,n}$ with disks and let $p_i$ denote the points at the center of these disks and denote the resulting closed surface by $\overline{S_{g,n}}$.  Let $G$ have a faithful action on some set $F$.  Then each element of $\Hom^{(C_1,...,C_n)}( \pi_1(S_{g,n}), G)$ gives rise to a (not necessarily connected) branched covering of $\overline{S_{g,n}}$ with Galois group $G$ and ramification points $p_i$.  Conversely, each such branched covering comes from  $\Hom^{(C_1,...,C_n)}( \pi_1(S_{g,n}), G)$.  So equation (\ref{eq:closed_bound}) gives a method for computing the number of such branched coverings.  The construction extends analogously to branched covers of nonorientable surfaces and Theorem \ref{thm:main} then supplies the relevant counting formula.

As a closing remark for this section, we mention another family of words $\gamma$ where there is a nice formula for $a_\chi^\gamma$ for all $G$ and $\chi$.  Let 
$$
[x_1,...,x_m] = x_1 \cdots x_m x_1^{-1} \cdots x_m^{-1}
$$
(these are sometimes referred to as generalized commutators).  Then a formula of Leitz \cite{leitz} (see also \cite{tambour}) says that 
$$
a_\chi^\gamma = \frac{|G|^{m-1}}{\chi(1)^{m-\epsilon_m}}
$$
where
$$
\epsilon_m = 
\begin{cases} 
	1 & \text{if $m$ is even} \\
        2 & \text{if $m$ is odd}
   \end{cases}
$$
Using this together with Proposition \ref{prod}, a formula for the $f_\gamma$ follows where $\gamma$ is a product of generalized commutators just as in Theorem \ref{MulaseYu} (and this generalizes the formula for products of commutators in Theorem \ref{MulaseYu}). Similarly, applying Theorem \ref{thm:general} with these words, we obtain a result generalizing Equation (\ref{eq:closed_bound}) (although with no apparent topological interpretation).

In a similar vein but generalizing the case of the elements $ x_1^2 x_2^2 \cdots x_k^2$, we can consider words of the form $\gamma = x^n$ in a free group of rank one generated by $x$.  In this case, the class function $f_\gamma$ counts the number of $n$th roots that each element in $G$ has.  Therefore, in this case, we have 
$$
a_\gamma^\gamma = \nu_n(\chi)
$$
where $\nu_n$ is the gereralized Frobenius-Schur indicator given by
$$
\nu_n(\chi) = \frac{1}{|G|} \sum_{g \in G} \chi(g^n)
$$
(see, for example, Chapter 4 of \cite{isaacs}).  In general, $\nu_n \in \mathbb{Z}$ so $f_\gamma$ is a virtual character.  From Proposition \ref{prod}, taking $\gamma = x_1^{n_1} x_2^{n_2} \cdots x_m^{n_m}$ in a free group of rank $m$ generated by $x_1,...,x_m$ with $n_1,...,n_m \in \mathbb{Z}$ (note that, in gereral, $f_\gamma = f_{\gamma^{-1}}$), we have 
$$
a_\gamma^\chi = \left( \frac{|G|}{\chi(1)}  \right)^{m-1} \nu_{n_1}(\chi) \cdots \nu_{n_m}(\chi)
$$
We could also involve conjugacy classes by applying Theorem \ref{thm:general}.  Similarly, we could take products of generalized commutators and powers of elements and do this again, obtaining a common generalization of all of the results.

\section{Some symmetric function identities} \label{sec:combinatorics}

In this section, we apply some of the identities from the previous sections to the symmetric group and use an isomorphism between the set of class functions on the symmetric group with a certain space of symmetric functions in order to obtain a few identities among symmetric functions.  Let $S_n$ be the symmetric group on an $n$ element set.  The relevant definitions and background come from \cite{stanley}.  Let $\Lambda$ be the ring of symmetric functions over the complex numbers and let $\Lambda^n$ be the subspace of $\Lambda$ spanned by symmetric functions of degree $n$ so that $\Lambda$ has a grading 
$$
\Lambda = \bigoplus_{n \geq 0} \Lambda^n
$$
Given a partition $\lambda$ of $n$, denoted $\lambda \vdash n$, let $p_\lambda$ and $s_\lambda$ be the associated power and Schur symmetric functions.  Let $\chi^\lambda$ be the irreducible character of $S_n$ coming from the Specht module $S^\lambda$.  Let $H_\lambda$ denote the product of the hook lengths of the Young tableau associated to $\lambda$.  Given an element $w \in S_n$, let $\rho(w)$ denote the partition of $n$ given by cycle type of $n$.  Let $R^n$ denote the set of class functions on $s_n$, then we have a vector space isomorphism
\begin{align*}
	\ch : R^n &\to \Lambda^n \\
	      f   &\mapsto \frac{1}{n!} \sum_{w \in S_n} f(w) p_{\rho(w)}
\end{align*}
and this has the property that $\ch(\chi^\lambda) = s_\lambda$.

Given $\gamma \in F_r$, then we have the class function $f_\gamma$ and 
$$
\ch(f) = \frac{1}{n!} \sum_{u_1,...,u_r \in S_n} p_{\rho(u_1,...,u_r)}
$$
and using the expansion 
$$
f_\gamma = \sum_{\lambda \vdash n} a_{\chi^\lambda}^\gamma \chi^\lambda
$$
we also have
$$
\ch(f) = \sum_{\lambda \vdash n} a_{\chi^\lambda}^\gamma s_\lambda
$$
and therefore, we have shown:

\begin{theorem}
Given $\gamma \in F_r$,
$$
\frac{1}{n!} \sum_{u_1,...,u_r \in S_n} p_{\rho(u_1,...,u_r)} = \sum_{\lambda \vdash n} a_{\chi^\lambda}^\gamma s_\lambda
$$
\end{theorem}

This is a generalization of exercise 7.68 (c) in \cite{stanley}.  Taking the examples of $x_1^2x_2^2 \cdots x_k^2 \in F_k$ and $[x_1,y_1]\cdots[x_g,y_g] \in F_{2g}$ and applying the hook length formula, which says that 
$$
\chi^\lambda(1) = \frac{n!}{H_\lambda}
$$
together with the fact that all of the Frobenius-Schur indicators for the symmetric group are $1$, we have:

\begin{corollary}
For integers $n,k,g \geq 1$, we have 
$$
	\frac{1}{n!} \sum_{u_1,...,u_k \in S_n} p_{\rho(u_1^2u_2^2 \cdots u_k^2)} = \sum_{\lambda \vdash n} H_\lambda^{k-1} s_\lambda
$$
and 
$$
	\frac{1}{n!} \sum_{u_1,v_1,...,u_g,v_g \in S_n} p_{\rho([u_1,v_1] \cdots [u_g,v_g])} = \sum_{\lambda \vdash n} H_\lambda^{2g-1} s_\lambda
$$
\end{corollary}

Let $1^q$ denote the vector that is $q$ ones followed by zeroes.  Recalling that 
$$
p_{\rho(w)}(1^q) = q^{\kappa(w)}
$$
where $\kappa(w)$ is the number of cycles in $w$ and that 
$$
s_\lambda(1^q) = H_\lambda^{-1} \prod_{t \in \lambda} (q + c(t))
$$
where $c(t)$ denotes the content of $\lambda$ at $t$, then by specializing to $1^q$, we obtain:

\begin{corollary}
For integers $n,k,g \geq 1$, we have 
$$
	\frac{1}{n!} \sum_{u_1,...,u_k \in S_n} q^{\kappa(u_1^2u_2^2 \cdots u_k^2)} = \sum_{\lambda \vdash n} H_\lambda^{k-2} \prod_{t \in \lambda} (q + c(t))
$$
and 
$$
	\frac{1}{n!} \sum_{u_1,v_1,...,u_g,v_g \in S_n} q^{\kappa([u_1,v_1] \cdots [u_g,v_g])} = \sum_{\lambda \vdash n} H_\lambda^{2g-2} \prod_{t \in \lambda} (q + c(t))
$$
\end{corollary}

The first of these is a generalization of exercise 7.68 (e) of \cite{stanley}.  More such identities can be obtained by applying the same ideas to the identities mentioned at the end of the last section.  

As a final remark, fix a word $\gamma = x_1^2x_2^2 \cdots x_k^2 \in F_k$ or $[x_1,y_1]\cdots[x_g,y_g] \in F_{2g}$ and consider the sequence $f_{\gamma, S_n}$ (here we have made the group explicitly a part of the notation) as $n$ varies.  The exponential generating function for this sequence in the case where $\gamma = x_1^2x_2^2 \cdots x_k^2 \in F_k$ is, by  Corollary \ref{MulaseYu} together with the hook length formula,
$$
\sum_{n \geq 0 } \sum_{\lambda \vdash n} \left( \frac{n!}{H_\lambda} \right)^{k-2} x^n
$$
which in the specific case of $k=2$ has the particularly nice form
\begin{equation} \label{product}
\prod_{i \geq 1} (1 - x^i)^{-1}
\end{equation}
(see \cite{stanley} exercise 5.12).  Similarly for $\gamma = [x_1,y_1]\cdots[x_g,y_g] \in F_{2g}$, we have the exponential generating function
$$
\sum_{n \geq 0 } \sum_{\lambda \vdash n} \left( \frac{n!}{H_\lambda} \right)^{2g-2} x^n
$$
which for $g = 1$ again gives the product in Equation (\ref{product}).  

\bibliography{character_identity}
\bibliographystyle{alpha}

\end{document}